\documentclass[12pt]{article}
\usepackage{amsmath, amssymb, amscd, amsthm, amsfonts, xfrac, mathrsfs,diagbox}
\numberwithin{equation}{section}
\usepackage{graphicx}
\usepackage{hyperref}
\usepackage{authblk}
\usepackage{url}
\hypersetup{backref, colorlinks=true}
\newtheorem{theorem}{Theorem}[section]

\newtheorem{lemma}[theorem]{Lemma}

\newcommand{\sq}{\sqrt[n]{a_n}}

\title{\textbf{Inequalities associated with the root sequences of P-recursive sequences}}
\author{Zhongjie Li}
\affil{School of Mathematics \break Tianjin University \break Tianjin 300072, China \break\texttt{lizhongjie@tju.edu.cn}}
\date{}

\begin{document}
\maketitle
\begin{abstract}
The Tur{\'a}n inequalities and the Laguerre inequalities are closely related to the Laguerre-P\'{o}lya class and the Riemann hypothesis. These inequalities have been extensively studied in the literature. In this paper, we propose a method to determine a positive integer $N$ such that the sequences $\{\sq/n!\}_{n \ge N}$ and $\{\sqrt[n+1]{a_{n+1}}/(\sq n!)\}_{n \ge N}$ satisfy the higher order Tur{\'a}n inequalities and the Laguerre inequalities of order two for a P-recursive sequence $\{a_n\}_{n \ge 1}$.

\end{abstract}

\section{Introduction}
A sequence $\{a_n\}_{n \ge 1}$ of real numbers is said to satisfy the \emph{Tur{\'a}n inequalities} or to be \emph{log-concave} if for all $n \ge 2$,
\begin{align}\label{11}
a_{n}^2 \ge a_{n-1}a_{n+1}.
\end{align}
The sequence $\{a_n\}_{n \ge 1}$ is said to satisfy the \emph{higher order Tur{\'a}n inequalities} or the \emph{cubic Newton inequalities} if for all $n \ge 2$,
\begin{align}\label{12}
4(a_{n}^2 - a_{n-1}a_{n+1})(a_{n+1}^2 - a_{n}a_{n+2}) - (a_{n}a_{n+1} - a_{n-1}a_{n+2})^2 \ge 0.
\end{align}
The Tur{\'a}n inequality and the higher order Tur{\'a}n inequality are related to the Laguerre-P\'{o}lya class of real entire functions. A real entire function
\begin{align}\label{13}
\psi(x) = \sum_{n=0}^{\infty} a_{n} \frac{x^n}{n!}
\end{align}
is said to be in the Laguerre-P\'{o}lya class, denoted $\psi(x) \in \mathcal{LP}$, if 
\begin{align*}
\psi(x) = cx^{m}e^{-\alpha x^2 + \beta x}\prod \limits_{k=1}^{\infty}(1 + x/x_{k})e^{-x/x_{k}},
\end{align*}
where $c,\ \beta,\ x_{k}$ are real numbers, $\alpha \ge 0$, $m$ is a nonnegative integer and $\sum x_{k}^{-2} < \infty$. P\'{o}lya and Schur \cite{schur1914zwei} proved that if a real entire function $\psi(x) \in \mathcal{LP}$, its Maclaurin coefficients satisfy the Tur{\'a}n inequality (\ref{11}). And Dimitrov \cite{dimitrov1998higher} established that the Maclaurin coefficients of a real function $\psi(x) \in \mathcal{LP}$ satisfy the higher order Tur{\'a}n inequality (\ref{12}).\\
\indent In recent years, many combinatorial sequences have been demonstrated to satisfy the inequality (\ref{12}). Chen, Jia and Wang \cite{chen2019higher} proved that the partition function $p(n)$ satisfies the higher order Tur{\'a}n inequality when $n \ge 95$, which was conjectured by Chen~\cite{chen2017spt}. Griffin, Ono, Rolen and Zagier \cite{griffin2019jensen} showed that the partition function $p(n)$ satisfies the order $d \ge 1$ Tur{\'a}n inequalities for sufficiently large $n$. Liu and Zhang \cite{liu2021inequalities} proved that the overpartition function $\bar{p}(n)$ satisfies the higher order Tur{\'a}n inequalities for $n \ge 16$. Wang \cite{wang2019higher} presented a unified approach to studying the higher order Tur{\'a}n inequalities for the sequence $\{a_n/n!\}_{n \ge 0}$ when $a_{n}$ satisfies a second-order linear recurrence. Moreover, Hou and Li \cite{hou2021log} proposed a method to determine a specific $N$ such that the higher order Tur{\'a}n inequality holds for a P-recursive sequence $\{a_n\}_{n \ge N}$.\\
\indent A sequence $\{a_n\}_{n \ge 1}$ satisfies the Laguerre inequality of order $m$ if for $n \ge 1$,
\begin{align}\label{14}
L_{m}(a_{n}) := \frac{1}{2} \sum_{k=0}^{2m} (-1)^{k+m} \binom{2m}{k} a_{n+k} a_{2m-k+n} \ge 0.
\end{align}
A polynomial $f(x)$ satisfies the Laguerre inequality if
\begin{align*}
f'(x)^2 - f(x)f''(x) \ge 0.
\end{align*}
Laguerre \cite{laguerre1989} stated that the Laguerre inequality holds for $f(x)$ if $f(x)$ is a polynomial with only real zeros. Jensen \cite{jensen1913recherches} introduced the $n$-th generalization of the Laguerre inequality as
\begin{align}\label{15}
L_{n}(f(x)) := \frac{1}{2} \sum_{k=0}^{2n} (-1)^{n+k} \binom{2n}{k} f^{(k)}(x) f^{(2n-k)}(x) \ge 0,
\end{align}
where $f^{(k)}(x)$ denotes the $k$th derivative of $f(x)$. By choosing the function $f(x)$ to have Taylor coefficients $a_{n+m}$ through the specialization $x = 0$, the above inequality (\ref{15}) transforms into the inequality (\ref{14}).\\
\indent Wang and Yang \cite{wang2022laguerre} established that the partition function $p(n)$, the overpartition function $\bar{p}(n)$ and several other combinatorial sequences satisfy the Laguerre inequalities of order two. Dou and Wang \cite{dou2023higher} found $N(m)$ for $3 \le m \le 10$, such that the Laguerre inequalities of order $m$ hold for the partition function $p(n)$ when $n \ge N(m)$. Wagner \cite{wagner2022on} showed that the partition function $p(n)$ satisfies the Laguerre inequality of any order when $n$ is sufficiently large. Furthermore, Li \cite{li2022ell} gave a method to find the explicit integer $N$ such that the Laguerre inequality of order two holds for a P-recursive sequence $\{a_n\}_{n \ge N}$.\\
\indent In this paper, we provide the sufficient conditions for the sequences related to the root sequences to satisfy the higher order Tur{\'a}n inequalities and the Laguerre inequalities of order two. We introduce the concept of a \emph{root sequence} denoted by $\{\sq\}_{n \ge 1}$, which corresponds to a nonnegative real suquence $\{a_n\}_{n \ge 1}$. Also, we study the sequence $\{a_n\}_{n \ge 1}$ is a \emph{P-recursive  sequence}. Recall that a sequence $\{a_n\}_{n \ge 1}$ is called a P-recursive sequence of order $d$ if it satisfies a recurrence relation of the form
\begin{align*}
p_{0}(n)a_{n} + p_{1}(n)a_{n+1} + \dots + p_{d}(n)a_{n+d} = 0,
\end{align*}
where $p_{i}(n)$ are polynomials in $n$.\\
\indent We can solve the given problem by establishing upper and lower bounds on $u_n = \sqrt[n-1]{a_{n-1}}\sqrt[n+1]{a_{n+1}}/\sq^{2}$. To achieve this, we use the asymptotic expression of $u_n$. Hou and Li \cite{hou2023log} provided the asymptotic expansion of $\log {u_n}$, then we can derive the asymptotic expansion of $u_n$:
\begin{align*}
& 1 - \frac{\mu_{0}}{n^2} + \sum_{j=1}^{\rho-1}\frac{\mu_{j}(j/\rho-1)(j/\rho-2)}{n^{3-j/\rho}} + \frac{r(2 \log{n}-3)}{n^3} \\
& +\sum\limits_{s=0}^M\frac{\tilde{b}_{s}(1+s/\rho)(2+s/\rho)}{n^{3+s/\rho}} + \dots + o\left(\frac{1}{n^{1+M/\rho}}\right),
\end{align*}
where $\tilde{b}_0 =\log b_0$ and $\tilde{b}_i$ is a polynomial in $b_1/b_0, \ldots, b_i/b_0$.\\
\indent Based on the asymptotic expression of $a_n$, we can establish both upper and lower bounds for $a_n$ as well as for the ratio $a_{n+1}/a_{n}$. Similarly, according to the asymptotic form of $u_n$, we can derive upper and lower bounds for $u_n$. Consequently, we introduce a method to demonstrate that the sequence $\{\sq/n!\}_{n \ge N}$ and $\{\sqrt[n+1]{a_{n+1}}/(\sq n!)\}_{n \ge N}$ satisfies the higher order Tur{\'a}n inequalities and the Laguerre inequalities of order two. Finally, we provide examples to illustrate the proof method.\\
\indent This paper is organized as follows. In Section \ref{s2}, we begin by utilizing the bounds of $a_n$ and $a_{n+1}/a_{n}$ provided by \cite{hou2023log} to establish the criteria for calculating the bounds of $u_n$. Subsequently, we give a method to compute $N$, such that the sequences $\{\sq/n!\}_{n \ge N}$ and $\{\sqrt[n+1]{a_{n+1}}/(\sq n!)\}_{n \ge N}$ satisfy the higher order Tur{\'a}n inequalities. In Section \ref{s3}, we use the upper and lower bounds derived in Section \ref{s2} to establish sufficient conditions for proving that the sequences $\{\sq/n!\}_{n \ge N}$ and $\{\sqrt[n+1]{a_{n+1}}/(\sq n!)\}_{n \ge N}$ satisfy the Laguerre inequalities of order two.

\section{The higher order Tur{\'a}n inequality}\label{s2}

\,\,\,\,\,\,\,  In this section, we will first present a systematic approach to prove that the sequence $\left\{{\sq}/{n!}\right\}_{n \ge 1}$ satisfies the higher order Tur{\'a}n inequality when $\{a_n\}_{n \ge 1}$ is a P-recursive sequence.\\
\indent Hou and Li \cite{hou2021log} provided a method to prove that the higher order Tur{\'a}n inequalities hold for the P-recursive sequence $\{a_n\}_{n \ge 1}$.

\begin{lemma}\rm{\bf{(\cite[Theorem 5.2]{hou2021log})}}\label{l1}
Let
\begin{align}\label{21}
t(x,y)=4(1-x)(1-y)-(1-xy)^2.
\end{align}
If there exists an integer $N$ and two rational functions $f_n$ and $g_n$ about $n$, such that for all $n \ge N$,
\[f_n < a_{n-1}a_{n+1}/a_{n}^2 < g_n\]
and
\[t(f_{n},f_{n+1}) > 0,\ t(f_{n},g_{n+1}) > 0,\ t(g_{n},f_{n+1}) > 0,\ t(g_{n},g_{n+1}) > 0.\]
Then $\{a_n\}_{n \ge N}$ satisfies the higher order Tur{\'a}n inequality.
\end{lemma}

We will use the lemma above to prove the sequence $\{\sq/n!\}_{n \ge 1}$ satisfies the higher order Tur{\'a}n inequalities. To accomplish this, we need to determine the lower and upper bounds of $u_n = \sqrt[n-1]{a_{n-1}}\sqrt[n+1]{a_{n+1}}/\sq^2$. Additionally, we will provide a criterion for establishing these bounds.

\begin{theorem}\label{t2}
Let $\{a_n\}_{n \ge 1}$ be a positive sequence. Suppose we can find a lower bound $s_n$ and an upper bound $S_n$ of $a_n$, a lower bound $f_n$ and an upper bound $g_n$ of $r_n = \frac{a_{n+1}}{a_n}$, and two rational functions $\widetilde{f_n}$ and $\widetilde{g_n}$, such that for $n \ge N$,
\begin{align*}
(n^2 - n)\log{f_{n}} - (n^2 + n - 2)\log{g_{n-1}} + 2\log{s_{n-1}} > (n - 1)n(n + 1)\log{\widetilde{f_{n}}},
\end{align*}
and
\begin{align*}
(n^2 - n)\log{g_{n}} - (n^2 + n - 2)\log{f_{n-1}} + 2\log{S_{n-1}} < (n - 1)n(n + 1)\log{\widetilde{g_{n}}}.
\end{align*}
Then we have $\widetilde{f_n} < u_n < \widetilde{g_n}$ for $n \ge N$.
\end{theorem}

\begin{proof}
In order to prove $\widetilde{f_n} < u_n < \widetilde{g_n}$, it suffices to prove 
\[\widetilde{f_n}^{(n - 1)n(n + 1)} < \frac{a_{n-1}^{n(n + 1)}a_{n+1}^{n(n - 1)}}{a_{n}^{2(n - 1)(n + 1)}} < \widetilde{g_n}^{(n - 1)n(n + 1)}.\]
Equivalently, we need to verify the following inequalities
\[(n - 1)n(n + 1)\log{\widetilde{f_n}} < n(n + 1)\log{a_{n-1}} + n(n - 1)\log{a_{n+1}} - 2(n - 1)(n + 1)\log{a_{n}}\]
and
\[n(n + 1)\log{a_{n-1}} + n(n - 1)\log{a_{n+1}} - 2(n - 1)(n + 1)\log{a_{n}} < (n - 1)n(n + 1)\log{\widetilde{g_n}}.\]
By substituting $a_{n} = a_{n-1}r_{n-1}$ and $a_{n+1} = a_{n}r_{n}$ into the above two inequalities, we obtain
\[(n - 1)n(n + 1)\log{\widetilde{f_n}} < (n^2 - n)\log{r_{n}} - (n^2 + n - 2)\log{r_{n-1}} + 2\log{a_{n-1}}\]
and
\[(n^2 - n)\log{r_{n}} - (n^2 + n - 2)\log{r_{n-1}} + 2\log{a_{n-1}} < (n - 1)n(n + 1)\log{\widetilde{g_n}}.\]
Consequently, the theorem is proven immediately.
\end{proof}

Using the above lower and upper bounds and Lemma \ref{l1}, we can establish the criteria that enable the sequence $\{\sq/n!\}_{n \ge 1}$ to satisfy the higher order Tur{\'a}n inequality.

\begin{theorem}\label{t3}
Let $\{a_n\}_{n \ge 1}$ be a positive sequence,
\[p_{n} = \frac{n}{n+1} \widetilde{f_n},\ q_{n} = \frac{n}{n+1} \widetilde{g_n}.\]
If there exists an integer $N$ such that for $n \ge N$,
\[t(p_{n},p_{n+1}) > 0,\ t(p_{n},q_{n+1}) > 0,\ t(q_{n},p_{n+1}) > 0,\ t(q_{n},q_{n+1}) > 0.\]
Then the sequence $\{\sq/n!\}_{n \ge N}$ satisfies the higher order Tur{\'a}n inequality.
\end{theorem}

Similarly, we can provide the sufficient conditions for the sequence \\
$\left\{\sqrt[n+1]{a_{n+1}}/(\sq n!)\right\}_{n \ge 1}$ to satisfy the higher order Tur{\'a}n inequality.

\begin{theorem}\label{t4}
Let $\{a_n\}_{n \ge 1}$ be a positive sequence,
\[\widetilde{p_{n}} = \frac{n}{n+1} \frac{\widetilde{f_{n+1}}}{\widetilde{g_{n}}},\ \widetilde{q_{n}} = \frac{n}{n+1} \frac{\widetilde{g_{n+1}}}{\widetilde{f_{n}}}.\]
If there exists an integer $N$ such that for $n \ge N$,
\[t(\widetilde{p_{n}},\widetilde{p_{n+1}}) > 0,\ t(\widetilde{p_{n}},\widetilde{q_{n+1}}) > 0,\ t(\widetilde{q_{n}},\widetilde{p_{n+1}}) > 0,\ t(\widetilde{q_{n}},\widetilde{q_{n+1}}) > 0.\]
Then the sequence $\left\{\sqrt[n+1]{a_{n+1}}/(\sq n!)\right\}_{n \ge N}$ satisfies the higher order Tur{\'a}n inequality.
\end{theorem}

Based on the Lemma \ref{l1}, the proof of Theorems \ref{t3} and \ref{t4} is obvious. So we omit the proof process. Next, we will provide an example to illustrate the application of the above theorems.

\begin{theorem}\rm{\bf{(\cite[Conjecture 4.7]{zhao2023inequalities})}}
Let
\[B_n = \sum_{k = 1}^{n} \frac{2}{n(n + 1)^{2}} \binom{n+1}{k-1} \binom{n+1}{k} \binom{n+1}{k+1}\]
be the Baxter number. Then the sequences $\{\sqrt[n]{B_n}/n!\}_{n \ge 2}$ and\\
$\left\{\sqrt[n+1]{B_{n+1}}/(\sqrt[n]{B_n} n!)\right\}_{n \ge 2}$ satisfy the higher order Tur{\'a}n inequalities.
\end{theorem}

\begin{proof}
By Zeilberger's algorithm, we get that $B_n$ satisfies the following recurrence relation,
\[(n+3)(n+4)B_{n+1} = (7n^2+21n+12)B_n + 8n(n-1)B_{n-1}\]
with initial values
\[B_{0} = 1,\ B_{1} = 1,\ B_{2} = 2,\ B_{3} = 6.\]
By utilizing algorithm {\tt Asy} from the {\tt Mathematica} package {\tt P-rec.m}, we can obtain the asymptotic expansion of $B_n$,
\[C\cdot{8^{n}{n^{-4}}\left(1 - \frac{22}{3n} + \frac{955}{27n^2} + o\left(\frac{1}{n^2}\right)\right)},\]
where $C$ is a constant.\\
\indent Then employing the algorithm {\tt RootLog}, we derive a lower bound $f_n$ and an upper bound $g_n$ for $r_n = {B_{n+1}}/{B_{n}}$ when $n \ge 753$,
\[f_n = 8 - \frac{32}{n} + \frac{413}{3n^2},\ g_n = 8 - \frac{32}{n} + \frac{419}{3n^2}.\]
Next, we will demonstrate by mathematical induction that $s_n = 8^{n}n^{-5}$ is a lower bound and $S_n = 8^{n}n^{-3}$ is an upper bound for $B_n$ when $n \ge 3$. Assuming that $s_n < B_n < S_n$, then we will prove $s_{n+1} < B_{n+1} < S_{n+1}$. On the one hand, we have 
\[B_{n+1} = r_{n}B_{n} > f_{n}s_{n}\]
and
\[\frac{f_{n}s_{n} - s_{n+1}}{s_n} = \frac{413+1969n+3674n^2+3290n^3+1345n^4+173n^5+24n^6}{3n^2(1+n)^5},\]
which is positive for $n \ge 1$. On the other hand, we have
\[B_{n+1} = r_{n}B_{n} < g_{n}S_{n}\]
and
\[\frac{S_{n+1}- g_{n}S_{n}}{S_n} = \frac{-419-1161n-993n^2-203n^3+24n^4}{3n^2(1+n)^3},\]
which is positive for $n \ge 13$. Therefore, we can conclude that $s_n < B_n < S_n$ for $n \ge {753}$. Checking the first 752 items, we finally get that 
\[s_n < B_n < S_n, \quad \forall\, n\ge{3}.\]
By utilizing the asymptotic expansion of $B_n$, we obtain the asymptotic expression of $u_n = {\sqrt[n-1]{B_{n-1}}\sqrt[n+1]{B_{n+1}}}/{\sqrt[n]{B_n}^2}$ as follows,
\[1 + \left(\frac{12}{n^3} - \frac{8}{n^3} \log{n}\right) + o\left(\frac{1}{n^3}\right).\]
Consequently, we will prove using Theorem \ref{t2} that $\widetilde{f_n}=1-\frac{1}{n^2}$ serves as a lower bound and $\widetilde{g_n}=1-\frac{8}{n^3}$ acts as an upper bound for $u_n$ when $n \ge 14$.
Let
\[D_{1}(n) = (n^2 - n)\log{f_{n}} - (n^2 + n - 2)\log{g_{n-1}} + 2\log{s_{n-1}} - (n - 1)n(n + 1)\log{\widetilde{f_{n}}}\]
and
\[D_{2}(n) = (n - 1)n(n + 1)\log{\widetilde{g_{n}}} - (n^2 - n)\log{g_{n}} + (n^2 + n - 2)\log{f_{n-1}} - 2\log{S_{n-1}}.\]
Hence, we need to prove that $D_{1}(n) > 0$ and $D_{2}(n) > 0$ when $n \ge 14$. We know that $D_{1}^{(4)}(n)$ is a rational function of $n$ and based on the largest real root of the numerator of $D_{1}^{(4)}(n)$, we get that $D_{1}^{(4)}(n) > 0$ for $n \ge 32$. Then we obtain the following formulae by {\tt Mathematica},
\[\lim_{n \to +\infty} D_{1}(n) = \infty, \quad \lim_{n \to +\infty} D_{1}^{'}(n) = 1, \quad \lim_{n \to +\infty} D_{1}^{''}(n) = \lim_{n \to +\infty} D_{1}^{'''}(n) = 0.\]
Thus, we deduce that
\[D_{1}^{'''}(n) < 0, \quad D_{1}^{''}(n) > 0, \quad D_{1}^{'}(n) > 0, \quad \forall\, n\ge{32}.\]
Since $D_{1}(32) > 0$, we conclude that $D_{1}(n) > 0$ for $n \ge 32$. By a similiar proof process, we can show that $D_{2}(n) > 0$ for $n \ge 44$. Consequently, we have $\widetilde{f_n} < u_n <\widetilde{g_n}$ for $n \ge 44$. By verifying the initial values, we derive that
\[\widetilde{f_n} < u_n < \widetilde{g_n}, \quad \forall\, n\ge{14}.\]
Then when $n \ge 14$, let
\begin{align*}
p_n = \frac{n}{n+1}\widetilde{f_n} = \frac{n}{n+1}\left(1 - \frac{1}{n^2}\right),\\
q_n = \frac{n}{n+1}\widetilde{g_n} = \frac{n}{n+1}\left(1 - \frac{8}{n^3}\right).
\end{align*}
By direct calculation, we find that
\[t(p_n,p_{n+1}) = \frac{4}{n{(1+n)}^2},\]
which is positive for $n \ge 1$.
\[t(p_n,q_{n+1}) = \frac{-49+240n+122n^2+100n^3+15n^4+4n^5}{n^2(1+n)^4(2+n)^2},\]
which is positive for $n \ge 1$.
\[t(q_n,p_{n+1}) = \frac{-32+48n+3n^2+4n^3}{n^2(1+n)^4},\]
which is positive for $n \ge 1$.
\begin{align*}
t(q_n,q_{n+1}) &= \frac{4}{{n^4(1+n)^6(2+n)^2}}(-784+672n+728n^2+840n^3+529n^4\\
&+ 325n^5+98n^6+34n^7+5n^8+n^9),
\end{align*}
which is positive for $n \ge 1$.\\
\indent Therefore, we can conclude that the sequence $\{\sqrt[n]{B_n}/n!\}_{n \ge 14}$ satisfies the higher order Tur{\'a}n inequality according to Theorem \ref{t3}. By examining the initial terms, we can further deduce that sequence $\{\sqrt[n]{B_n}/n!\}_{n \ge 2}$ satisfies the higher order Tur{\'a}n inequality.\\
\indent Next when $n \ge 14$, let
\begin{align*}
\widetilde{p_n} &= \frac{n}{n+1}\frac{\widetilde{f_{n+1}}}{\widetilde{g_n}} = \frac{{n}^5(2+{n})}{(1+{n})^3\left(-8+{n}^3\right)},\\
\widetilde{q_n} &= \frac{n}{n+1}\frac{\widetilde{g_{n+1}}}{\widetilde{f_n}} = \frac{{n^3(7+4n+n^2)}}{(1+n)^5}.
\end{align*}
So we can get that
\begin{align*}
& t(\widetilde{p_n},\widetilde{p_{n+1}}) \\
={} & \frac{4}{(-2 + n)^2(-1 + n)^2 (1 + n)^3 (2 + n)^4 (4 + 2n + n^2)^2 (7 + 4n + n^2)^2}(40320 \\
+{} & 178496n + 290816n^2 + 145312n^3 - 177824n^4- 322592n^5 - 192458n^6\\
-{} & 18235n^7 + 39578n^8 + 20340n^9 + 35n^{10} - 3304n^{11} - 1083n^{12} + 71n^{14}\\
+{} & 15n^{15} + n^{16}) > 0,\quad \forall\, n\ge{4}.
\end{align*}
\begin{align*}
& t(\widetilde{p_n},\widetilde{q_{n+1}}) \\
={} & \frac{4}{(-2 + n)^2 (1 + n)^3 (2 + n)^8 (4 + 2n + n^2)^2}(6144 + 36864n + 94720n^2 \\ 
+{} & 132864n^3 + 102144n^4 + 26944n^5 - 25056n^6 - 29952n^7 - 14232n^8 - 3036n^9\\ 
+{} & 168n^{10} + 283n^{11} + 81n^{12} + 12n^{13} + n^{14}) > 0,\quad \forall\, n\ge{4}.
\end{align*}
\begin{align*}
& t(\widetilde{q_n},\widetilde{p_{n+1}}) \\
={} & \frac{4}{(-1 + n)^2 (1 + n)^5 (2 + n)^6 (7 + 4n + n^2)}(360 + 2468n + 7086n^2 + 7107n^3\\
-{} & 1451n^4 - 8433n^5 - 6858n^6 - 2217n^7 + 38n^8 + 261n^9 + 88n^{10} + 14n^{11}\\
+{} & n^{12}) > 0,\quad \forall\, n\ge{3}.
\end{align*}
\begin{align*}
& t(\widetilde{q_n},\widetilde{q_{n+1}}) \\
={} & \frac{4}{(1 + n)^5 (2 + n)^{10}}(384 + 3456n + 14080n^2 + 31088n^3 + 41512n^4 + 36784n^5\\
+{} & 23084n^6 + 10531n^7 + 3536n^8 + 868n^9 + 151n^{10} + 17n^{11} + n^{12}) > 0,\\
\forall {} &\, n\ge{1}.
\end{align*}
Consequenctly, we know that the sequence $\{\sqrt[n+1]{B_{n+1}}/(\sqrt[n]{B_n} n!)\}_{n \ge 14}$ satisfies the higher Tur{\'a}n inequality by Theorem \ref{t4}. After verifying the first 13 items, we ultimately conclude that the sequence $\{\sqrt[n+1]{B_{n+1}}/(\sqrt[n]{B_n} n!)\}_{n \ge 2}$ satisfies the higher Tur{\'a}n inequality.
\end{proof}
By using the {\tt Mathematica} package {\tt P-rec.m} (which is available at \cite{pre}), we can show that the higher order Tur{\'a}n inequalities hold for the two sequences associated with the root sequence of many combinatorial sequences, such as Fine numbers $\{f_n\}_{n \ge 4}$, Motzkin numbers $\{M_n\}_{n \ge 2}$, Cohen numbers $\{C_{n}\}_{n \ge 2}$, large Schr{\"o}der numbers $\{S_n\}_{n \ge 2}$,  the numbers of the set of all tree-like polyhexes with $n + 1$ hexagons $\{h_n\}_{n \ge 2}$, the numbers of walks on cubic lattice with $n$ steps, starting and finishing on the $xy$ plane and never going below it $\{w_n\}_{n \ge 2}$, the numbers of $n \times n$ matrices with nonnegative entries and every row and column sum $2$ $\{t_n\}_{n \ge 2}$, Domb numbers $\{D_n\}_{n \ge 2}$ and so on. Then we list the lower and upper bounds of $u_n = \sqrt[n-1]{a_{n-1}}\sqrt[n+1]{a_{n+1}}/\sqrt[n]{a_n}^{2}$ for $n \ge N$.
\begin{table}[!h]
    \renewcommand{\arraystretch}{1.2}
    \caption{The lower and upper bounds}
    \label{table_example}
    \centering
    \begin{tabular}{|c|c|c|c|}
        \hline
        \diagbox{} & the\ lower\ bound & the\ upper\ bound & $N$ \\
        \hline
        $f_n$ & $1-\frac{1}{n^2}$ & $1-\frac{3}{n^3}$ & 6\\
        \hline
        $M_n$ & $1-\frac{1}{n^2}$ & $1-\frac{3}{n^3}$ & 11\\
        \hline
        $C_n$ & $1-\frac{1}{n^2}$ & $1-\frac{5}{n^3}$ & 9\\
        \hline
        $S_n$ & $1-\frac{1}{n^2}$ & $1-\frac{3}{n^3}$ & 9\\
        \hline
        $h_n$ & $1-\frac{1}{n^2}$ & $1-\frac{3}{n^3}$ & 6\\
        \hline
        $w_n$ & $1-\frac{1}{n^2}$ & $1-\frac{3}{n^3}$ & 26\\
        \hline
        $t_n$ & $1-\frac{2}{n^2}+\frac{1}{n^3}$ & $1-\frac{1}{n^2}$ & 5\\
        \hline
        $D_n$ & $1-\frac{1}{n^2}$ & $1-\frac{3}{n^3}$ & 6\\
        \hline
    \end{tabular}
\end{table}

\section{The Laguerre inequality of order 2}\label{s3}

In this section, we will give sufficient conditions for the sequences $\{\sq/n!\}_{n \ge 1}$ and $\{\sqrt[n+1]{a_{n+1}}/(\sq n!)\}_{n \ge 1}$  to satisfy the Laguerre inequality of order two, where the sequence $\{a_n\}_{n \ge 1}$ is a P-recursive sequence. Li \cite{li2022ell} presented a method to determine the explicit value of  $N$ such that Laguerre inequality of order two holds for $\{a_n\}_{n\ge N}$.

\begin{lemma}\rm{\bf{(\cite[Theorem 5.1]{li2022ell})}}\label{l31}
If there exists an integer $N$, two rational functions $f_n$ and $g_n$, such that for all $n \ge N$,
\[f_n < \frac{a_{n-1}a_{n+1}}{a_{n}^{2}} < g_n,\]
and
\[f_{n-1} f_{n}^{2} f_{n+1} - 4 g_{n} + 3 >0.\]
Then $\{a_n\}_{n \ge N}$ satisfies Laguerre inequality of order two.
\end{lemma}

We will utilize the lemma mentioned above to provide sufficient conditions for the sequences $\{\sq/n!\}_{n \ge 1}$ and $\{\sqrt[n+1]{a_{n+1}}/(\sq n!)\}_{n \ge 1}$ to satisfy the Laguerre inequality of order two. Similarly to the previous section, let $\widetilde{f_n}$ and $\widetilde{g_n}$ represent the lower and upper bounds of $u_n = {\sqrt[n-1]{a_{n-1}}\sqrt[n+1]{a_{n+1}}}/{\sq^2}$, respectively.

\begin{theorem}\label{t32}
Let $\{a_n\}_{n \ge 1}$ be a positive sequence and
\[p_n = \frac{n}{n+1} \widetilde{f_n},\ q_n = \frac{n}{n+1} \widetilde{g_n}.\]
If there exists an integer $N$ such that for $n \ge N$,
\[p_{n-1} p_{n}^2 p_{n+1} - 4 q_{n} + 3 > 0,\]
then the sequence $\{\sq/n!\}_{n \ge N}$ satisfies the Laguerre inequality of order two.
\end{theorem}

\begin{theorem}\label{t33}
Let $\{a_n\}_{n \ge 1}$ be a positive sequence and
\[\widetilde{p_n} = \frac{n}{n+1} \frac{\widetilde{f_{n+1}}}{\widetilde{g_{n}}},\ \widetilde{q_n} = \frac{n}{n+1} \frac{\widetilde{g_{n+1}}}{\widetilde{f_n}}.\]
If there exists an integer $N$ such that for $n \ge N$,
\[\widetilde{p_{n-1}}\widetilde{p_{n}}\widetilde{q_{n}}\widetilde{p_{n+1}} - 4\widetilde{q_{n}} + 3 > 0,\]
then the sequence $\{\sqrt[n+1]{a_{n+1}}/(\sq n!)\}_{n \ge N}$ satisfies the Laguerre inequality of order two.
\end{theorem}

By applying Lemma \ref{l31}, we can directly obtain the two theorems above, thus we omit the proof process. Subsequently, we will illustrate the application of the theorems with an example.

\begin{theorem}
Let $H_n$ denote the number of $n \times n$ $(0,1)$-matrices with row and column sum 2. The sequences $\{\sqrt[n]{H_{n}}/n!\}_{n \ge 1}$ and $\{\sqrt[n+1]{H_{n+1}}/(\sqrt[n]{H_n} n!)\}_{n \ge 2}$ satisfy the Laguerre inequalities of order two.
\end{theorem}
\begin{proof}
According to Zeilberger's algorithm, we have the following recursion relation for $H_n$,
\[2 H_{n} - 2n(n-1) H_{n-1} - n(n-1)^2 H_{n-2} = 0,\]
with initial values 
\[H_{0}=1,\ H_{1}=0,\ H_{2}=1,\ H_{3}=6.\]
Using the {\tt Mathematica} package {\tt P-rec.m}, we obtain the asymptotic expression of $H_n$, 
\[C\cdot {e}^{-2n}{n}^{\frac12+2n}\left(1-\frac5{24n}-\frac{47}{1152n^2}+o\left(\frac{1}{n^2}\right)\right),\]
where $C$ is a constant.\\
\indent Additionally, for $n \ge 5$, we can establish a lower bound $f_n$ and an upper bound $g_n$ for $r_n = {H_{n+1}}/{H_n}$ as follows,
\[f_n = n^2 + \frac{3n}{2} + \frac{3}{4} + \frac{1}{4n} -\frac{19}{16n^2},\ g_n = n^2 + \frac{3n}{2} + \frac{3}{4} + \frac{1}{4n} +\frac{13}{16n^2}.\]
Then we aim to prove that $s_n = e^{-2n} n^{\frac14+2n}$ serves as a lower bound and $S_n = e^{-2n} n^{1+2n}$ serves as an upper bound of $H_n$ when $n \ge 5$ by induction. Suppose that $s_n < H_n < S_n$, we need to prove $s_{n+1} < H_{n+1} < S_{n+1}$. Since that 
\[f_n s_n < H_{n+1} = r_n H_n < g_n S_n .\]
Thus, it is sufficient to prove
\[s_{n+1} < f_n s_n,\ g_n S_n < S_{n+1}.\]
For $n \ge 5$, we define
\[\Delta_{1}(n) = \log{f_n} + \log{s_n} - \log{s_{n+1}},\ \Delta_{2}(n) = \log{S_{n+1}} - \log{g_n} - \log{S_{n}}.\]
Consequently, we know that $\Delta_{1}^{''}(n)$ and $\Delta_{2}^{''}(n)$ are rational functions about $n$. By analyzing the largest roots of their numerators, we conclude that $\Delta_{1}^{''}(n) > 0$ and $\Delta_{2}^{''}(n) > 0$ for $n \ge 3$. By calculation, we also establish that
\[\lim_{n \to +\infty} \Delta_{1}^{'}(n) = \lim_{n \to +\infty} \Delta_{2}^{'}(n) = 0.\] 
Given that $\Delta_{1}(3) > 0$ and $\Delta_{2}(3) > 0$, we derive that $\Delta_{1}(n) > 0$ and $\Delta_{2}(n) > 0$ when $n \ge 5$. Then examining the initial values, we observe that for $n \ge 5$,
\[s_n < H_n < S_n.\]
From the asymptotic expansion of $H_n$, we can derive the asymptotic expansion of $u_n = \sqrt[n-1]{H_{n-1}} \sqrt[n+1]{H_{n+1}}/\sqrt[n]{H_n}^2$,
\[1 - \frac{2}{n^2} + \left(-\frac{3}{2n^3} + \frac{\log{n}}{n^3}\right) + o\left(\frac{1}{n^3}\right).\]
Following a similar proof process as above, we can establish that $\widetilde{f_n} = 1-\frac{2}{n^2}+\frac{1}{n^3}$ is a lower bound and $\widetilde{g_n} = 1-\frac{1}{n^2}$ is an upper bound for $u_n$ when $n \ge 5$.\\
\indent Therefore, when $n \ge 5$, let
\begin{align*}
p_n &= \frac{n}{n+1} \widetilde{f_n} = \frac{n}{n+1} \left(1 - \frac{2}{n^2} + \frac{1}{n^3}\right),\\
q_n &= \frac{n}{n+1} \widetilde{g_n} = \frac{n}{n+1} \left(1 - \frac{1}{n^2}\right).
\end{align*}
Through calculation, we can obtain that
\begin{align*}
& p_{n-1} p_{n}^2 p_{n+1} - 4 q_n + 3\\
={} & \frac{2+3n-14n^2-6n^3+62n^4+74n^5+26n^6+2n^7}{{n^4(1+n)^4(2+n)}} > 0,\quad \forall\, n\ge{1}.
\end{align*}
Thus, based on the Theorem \ref{t32}, it can be concluded that the sequence $\{\sqrt[n]{H_n}/n!\}_{n \ge 5}$ satisfies the Laguerre inequality of order two. After checking the first 4 items, we finally deduce that the sequence $\{\sqrt[n]{H_n}/n!\}_{n \ge 1}$ satisfies the Laguerre inequality of order two.\\
\indent Next, when $n \ge 5$, we define
\begin{align*}
\widetilde{p_n} &= \frac{n}{n+1} \frac{\widetilde{f_{n+1}}}{\widetilde{g_{n}}} = \frac{n^4(1+3n+n^2)}{(-1+n)(1+n)^5},\\
\widetilde{q_n} &= \frac{n}{n+1} \frac{\widetilde{g_{n+1}}}{\widetilde{f_{n}}} = \frac{{n}^5(2+{n})}{(1+{n})^3(1-2{n}+{n}^3)}.
\end{align*}
Consequently, we can get that
\begin{align*}
& \widetilde{p_{n-1}} \widetilde{p_{n}} \widetilde{q_{n}} \widetilde{p_{n+1}} - 4 \widetilde{q_n} + 3\\
={} & \frac{1}{(-2 + n) (-1 + n )(1 + n)^4 (2 + n)^4 (-1 + n +n^2)}(-96 - 336n -144n^2\\
+{} & 893n^3 + 1458n^4 + 597n^5 - 227n^6 - 192n^7 - 15n^8 + 6n^9),
\end{align*}
which is positive for $n \ge 8$. According to Theorem \ref{t33}, we know that the sequence $\{\sqrt[n+1]{H_{n+1}}/(\sqrt[n]{H_n} n!)\}_{n \ge 8}$ satisfies the Laguerre inequality of order two. By verifying the initial values, we can further conclude that the sequence $\{\sqrt[n+1]{H_{n+1}}/(\sqrt[n]{H_n} n!)\}_{n \ge 2}$ satisfies the Laguerre inequality of order two.
\end{proof}
The method described above allows us to prove that two sequences associated with the root sequence of Fine numbers $\{f_n\}_{n \ge 3}$ and two sequences related to the root sequence of other sequences in Section \ref{s2} satisfy the Laguerre inequalities of order two for $n \ge 1$.

\end{document}